\documentclass{tran-l}

\newtheorem{theorem}{Theorem}[section]
\newtheorem{corollary}[theorem]{Corollary}

\newtheorem{proposition}[theorem]{Proposition}
\theoremstyle{definition}
\newtheorem{definition}[theorem]{Definition}
\theoremstyle{definition}

\theoremstyle{definition}
\newtheorem{example}[theorem]{Example}

\numberwithin{equation}{subsection}

\begin{document}

\title[WEAKLY AND COMPLETELY NILARY IDEALS]
 {WEAKLY AND COMPLETELY NILARY IDEALS}

\author{O. A. AL-MALLAH}

\address{Department of Mathematics, College of Science, King Faisal\\ University, Al-Hasa, P.O. Box 380,
 Saudi Arabia}

\email{omarmallah@yahoo.com}

\author{H. M. AL-NOGHASHI}

\address{Department of Mathematics, College of Science, King Faisal\\ University, Al-Hasa, P.O. Box 380,
 Saudi Arabia}
\email{alnoghashi@hotmail.com}

\thanks{}

\thanks{}

\subjclass{Primary ???; Secondary ???}


\date{}

\dedicatory{}

\commby{}


\begin{abstract}
The present paper introduces and studies some new types of rings and ideals such as completely nilary rings ( resp. completely nilary ideals ), weakly nilary ideals. Some properties of each are obtained
and some characterizations of each type are given.
\end{abstract}

\maketitle

\section*{Introduction}

Let us say that an ideal $I$ in a ring $A$ is \textbf{\textit{completely
prime}} if $A/I$ is a domain. We say $I$ is \textit{\textbf{completely
semiprime}} if  for every $a\in A$ such that $a^n\in I$ for some
$n\in\mathbb{N}$ then $a\in I.$ For more about semiprime and completely
semiprime ideals see~\cite[p.~66]{McCoy1964} or~\cite{McCoy1973}
respectively.

Let $A$ be a ring and $I\unlhd A.$ The ideal $I$ is called a \textbf{\textit{(principally) right primary ideal}} if whenever $J$ and $K$ are (principal) ideals of $A$ with $JK\subseteq I,$ then either $J\subseteq I$ or $K^n \subseteq I$ for some positive integer $n$ depending on $J$ and $K$. The ideal $I$ is called a \textbf{\textit{(principally) nilary ideal}} if whenever $J$ and $K$ are (principal) ideals of $A$ with $JK\subseteq I$,  then either $J^{m}\subseteq I$ or $K^{n}\subseteq I$ for some positive integers $m$ and $n$ depending on $J$ and $K.$ $A$ is said to be a \textbf{\textit{(principally) right primary ring}} or  \textbf{\textit{(principal) nilary ring}} if the zero ideal is a (principal) right primary or a (principal) nilary ideal of $A,$ respectively. See~\cite[Definition~1.1]{Birkenmeier2013133}

Throughout this paper, $A$ is a ring, need not be commutative and not necessarily contains identity,
unless otherwise stated and now, we introduce the following definitions and giving some examples.

\section{COMPLETELY NILARY}

In this section we introduce a new class of rings, called completely nilary ideal.

\begin{definition}
We call an ideal $I$ of a ring $A$ a \textbf{\textit{completely nilary
ideal}}  if $a,b\in A$ with $ab\in I$ then $a^n\in I$ or $b^m\in I$ for some
$n,m\in\mathbb{N}$ and we call $A$ a \textit{\textbf{completely nilary ring}}
if the zero ideal of $A$ is a completely nilary ideal, equivalently if
$a,b\in A$ such that $ab=0,$ then $a$ is nilpotent or $b$ is nilpotent.
\end{definition}

These conditions is called completely nilary as their relation to the nilary
conditions is the same as the relation between completely prime and prime.
See~\cite{McCoy1973}.

Just as a nilpotent ring must be (p-)nilary, note that a nil ring will always be both completely nilary.
However nil rings need not be (p-)nilary. Now, we give the following example to
establish this fact.

If $S$ and $T$ are two simple nil rings which are not nilpotent(such rings always exist,
see~\cite{Smoktunowicz2002}), then $A=S\oplus T$ is a nil ring, identifying $S$ and $T$ as ideals of $A$
we have $ST=0.$ Since $S$ and $T$ are simple they are principal and since they are not nilpotent, $A$ is
not a p-nilary ring (and hence not a nilary ring), but on the other hand as $A$ is nil, each of its
elements is nilpotent, so $A$ is trivially a completely nilary ring. Now, it is the time to give some
characterizations of nilary ideals (nilary rings), p-nilary ideals (p-nilary rings) and completely nilary
ideals (completely nilary rings).

It is necessary to mention that every prime ring is a nilary ring, since if
$A$ is a prime ring and $I,J$ are ideals of $A$ such that $IJ=0$ and if $I$
is not nilpotent, then $I\neq0,$ so there exists $0\neq a\in I.$ Now, for all
$b\in J,$ we have $aAb\subseteq IJ=0$ and as $A$ is a prime ring, we get
$a=0$ or $b=0$ and as $a\neq0,$ we have $b=0,$ so that $J=0$ and thus $J$ is
nilpotent. Also, it can be shown that completely nilary rings are independent
from p-nilary rings and nilary rings. Now, we give the following example to
establish this fact.

It is known that $\mathbb{Z}_2$ is a prime ring and one can easily check that
$M_{2\times2}[\mathbb{Z}_2]$ is also prime and thus it is a nilary ring and
hence a p-nilary ring, but it is not a completely nilary ring since if we
take $I=\left[
                        \begin{array}{cc}
                          1 & 0\\
                          0 & 0 \\
                        \end{array}
                      \right], J=\left[
                        \begin{array}{cc}
                          0 & 0 \\
                          0 & 1 \\
                        \end{array}
                      \right]
\in M_{2\times2}[\mathbb{Z}_2],$ then clearly $IJ=0$ but neither $I$ is
nilpotent nor $J.$ Note that previous example also illustrates that
 right (left) primary does not imply completely right (left) primary.

\begin{proposition}
Let $A$ be a ring and $I\unlhd A.$ Then\\
$I$ is a completely prime ideal if and only if $I$ is a completely semiprime
and completely nilary ideal.
\end{proposition}

\begin{proof}
$(\Rightarrow)$ Clearly.

$(\Leftarrow)$ Let $a,b\in A$ with $ab\in I.$ Since $I$ is a completely
nilary ideal, then $a^n\in I$ or $b^m\in I$ for some $n,m\in\mathbb{N}.$
Since $I$ is a completely semiprime ideal, then $a\in I$ or $b\in I.$ Hence
$I$ is a completely prime ideal.
\end{proof}

\begin{proposition}Let $Q_1, . . ., Q_n$ be completely nilary ideals of $A$ such that $Q ^s_k\subseteq Q_1\cap\cdots\cap Q_n$ for some fixed $1\leq k\leq n$ and some positive integer
$s.$ Then $Q_1\cdots Q_n$ is a completely nilary ideal of $A.$ Hence
if $Q$ is a completely nilary ideal of $A,$ then so is $Q^n,$ and
$A/Q^n$ is a nilary ring for each positive integer $n.$
\end{proposition}

\begin{proof}
Assume that $a,b\in A$ such that $ab\in Q_1\cdots Q_n\subseteq Q_k.$ Since $Q_k$ is completely nilary, then either $a^t\in Q_k$ or $b^m\in Q_k$ for some $t,m\in\mathbb{N}.$ Thus  $a^{tns}\in Q_k^{ns}\subseteq Q_1\cdots Q_n$ or $b^{mns}\in Q_k^{ns}\subseteq Q_1\cdots Q_n.$
\end{proof}

\begin{proposition}\label{2.3} Let $A$ be a ring and $I\unlhd A.$  Then:\\
$I$ is a completely nilary ideal if and only if $A/I$ is a completely
nilary ring.
\end{proposition}

\begin{proof}
Suppose $A/I$ is a completely nilary ring. Let $a, b\in A,$ with $ab\in I.$
Therefore $(a+I)(b+I)\in I,$ so $[(a+I)/I][(b+I)/I]=0.$ Hence either
$[(a+I)/I]^n=0$ for some positive integer $n$ or $[(b+I)/I]^m=0$ for some
positive integer $m.$ Then either $a^n\in I$ or $b^m\in I.$ Hence $I$ is a completely nilary ideal of $A.$
\end{proof}

\begin{proposition}\label{surjective:1}
Let $\phi: A\rightarrow\bar{A}$ be a surjective homomorphism, and let
$I\unlhd A$ with
$Ker(\phi)\subseteq I.$ Then:\\
If $I$ is completely nilary in $A,$ then $\phi(I)$ is completely nilary
of $\bar{A}.$
\end{proposition}

\begin{proof}
Let $K=Ker(\phi),$ $\bar{A}=A/K.$  (i) Since $I$ is a completely nilary ideal of $A,$
then $A/I$ is a completely nilary ring, by Proposition~\ref{2.3}. Now, since
$A/I\cong(A/K)/(I/K)$ and $A/I$  is a completely nilary ring, then $(A/K)/(I/K)$ is a
completely nilary ring. Thus, $I/K$ is a completely nilary ideal of $A/K,$ by
Proposition~\ref{2.3}. But $\phi(I)=I/K.$ Hence $\phi(I)$ is a completely nilary ideal
in $\bar{A}.$
\end{proof}

\begin{proposition}\label{surjective:2}
Let $\phi: A\rightarrow\bar{A}$ be a surjective homomorphism and
let $I'\unlhd\bar{A}$ with $I=\phi^{-1}(I').$ Then:\\
If $I'$ is completely nilary in $\bar{A},$ then $I$ is completely
nilary of $A.$
\end{proposition}

\begin{proof}
Let $\bar{A}=A/K,$  (i) Since $I=\phi^{-1}(I')$  and $\phi$ is a surjective
then $\phi(I)=\phi(\phi^{-1}(I'))=I'$ implies that $\phi(I)=I'.$ Now, let
$a,b\in A$ with $ab\in I$ then
$\phi(ab)\in\phi(I)=I'$ implies that $\phi(a)\phi(b)\in
I'$ and since $I'$ is a completely nilary ideal of $A/K$ then either
$\phi(a)^n\in I'$ or $\phi(b)^m\in I'$ for some
$n,m\in\mathbb{N}.$ If $\phi(a)^n\in I'.$ This implies that
$\phi(a^n)\in I'$ hence
$\phi^{-1}(\phi(a^n))\in\phi^{-1}(I')$ implies that
$a^n\in\phi^{-1}(\phi(a^n))\subseteq\phi^{-1}(I')=I.$ Hence
$a^n\in I.$ Similarly, If $\phi(b)^m\in I'$ then
$b^m\subseteq I.$ Therefore $I$ is completely nilary in $A.$
\end{proof}

\begin{corollary}Let $A$ be a ring and $I,K\unlhd A$ with $K\subseteq I.$\\
$I$ is a completely nilary ideal of $A$ iff $I/K$ is a
completely nilary ideal of $A/K.$
\end{corollary}

\begin{proof}
($\Rightarrow$) See Proposition~\ref{surjective:1}.

($\Leftarrow$) Assume that $I/K$ is a completely nilary ideal of $A/K.$ Put
$I'=I/K.$ By Proposition~\ref{surjective:2} then $\phi^{-1}(I')=I.$ Hence $I$
is a completely nilary ideal of $A.$
\end{proof}

\begin{proposition} Let $A$ be a ring and $I\unlhd A.$\\
$A/I$ is a completely nilary ring and $I$ is nil, then  $A$ is a completely
nilary ring.
\end{proposition}

\begin{proof}
Let $a,b\in A$ with $ab=0.$ Then $ab\in I.$ Since $A/I$ is a completely
nilary ring, then $I$ is a completely nilary ideal, by
Proposition~\ref{2.3}(iii). Hence $a^n\in I$ or $b^m\in I$ for some
$n,m\in\mathbb{N}.$ Since $I$ is nil, then $a^n$ and $b^m$ are nilpotent,
implies that $a$ and $b$ are nilpotent. Hence $A$ is a completely nilary
ring.
\end{proof}

\begin{proposition}\label{completely:p-nilary}
If $A$ is a commutative ring, then $A$ is a p-nilary ring if and only if it
is a completely nilary ring.
\end{proposition}

\begin{proof}($\Rightarrow$)
Suppose  $A$ is a p-nilary ring and that $a, b\in A$ with $ab=0.$ Since $A$
is commutative we get $\langle a\rangle\langle b\rangle=\langle ab\rangle=0$
and since $A$ is a p-nilary ring so we get $\langle a\rangle$ is nilpotent or
$\langle b\rangle$ is nilpotent. If $\langle a\rangle$ is nilpotent then
$\langle a\rangle^n=0$ for some $n\in\mathbb{N},$ so $a^n\in\langle
a\rangle^n=0,$ that means $a$ is nilpotent and if $\langle b\rangle$ is
nilpotent by the same argument we get $b$ is nilpotent and thus $A$ is a
completely nilary ring.

($\Leftarrow$)
 Now suppose $A$ is completely nilary and let $c, d\in A$ such that
$\langle c\rangle\langle d\rangle=0.$ Then $cd\in\langle c\rangle\langle
d\rangle=0,$ and since $A$ completely nilary implies $c$ or $d$ is nilpotent,
say $c^n=0$ or $c^m=0$ for some $n,m\in\mathbb{N}.$ If $c^n=0$ then $\langle
c\rangle^n= 0$ because $A$ is commutative. Similarly, if $d^m=0$ then
$\langle d\rangle^m=0.$ Thus $A$ is a p-nilary ring.
\end{proof}

\begin{corollary}
Let $A$ be a ring and $I$ is an ideal of $A.$ If $A/I$ is commutative then
$I$ is a p-nilary ideal if and only if $I$ is a completely nilary ideal.
\end{corollary}

\begin{proof}
By Proposition~\ref{2.3}(ii), then $I$ is a p-nilary ideal, iff $A/I$ is a
p-nilary ring. By Proposition~\ref{completely:p-nilary}, $A/I$ is a p-nilary
ring iff $A/I$ is a completely nilary ring. By Proposition~\ref{2.3}(iii),
$A/I$ is a completely nilary ring iff $I$ is a completely nilary ideal.
\end{proof}

\begin{proposition}\label{Jabbar:Th. 2.20}
Let $A$ be a ring. If $A$ is a completely nilary ring in which all nil ideals
are nilpotent then it is a nilary ring.
\end{proposition}

\begin{proof}
Let $I$ and $J$ be two ideals of $A$ with $IJ=0$ and suppose that $I$ is not
nilpotent, so by the hypothesis $I$ is not nil which implies that $I$
contains an element $a$ which is not nilpotent. Now, if $b\in J$ is any
element, then $ab\in IJ=0,$ so that $ab=0$ and as $A$ is a completely nilary
ring and $a$ is not nilpotent we must have $b$ is nilpotent. That means every
element of $J$ is nilpotent and so $J$ is a nil ideal and hence it is a
nilpotent ideal so that $A$ is a nilary ring.
\end{proof}

Since a nil left (resp. a nil right) ideal of a left (resp. right) Noetherian
ring is nilpotent and also a nil left (resp. a nil right) ideal of a left
(resp. right) Artinian ring is nilpotent~\cite{Basic1999}, so by combining
these facts with Proposition~\ref{Jabbar:Th. 2.20} we can give the following
corollary.

\begin{corollary}
Let $A$ be a ring. If $A$ is a completely nilary ring and satisfies any one
of the following
conditions then it is a nilary ring.\\
(1) $A$ is a left (resp. a right) Artinian.\\
(2) $A$ is a left (resp. a right) Noetherian.
\end{corollary}

\begin{corollary}
Let $A$ be a ring. If $A$ is a p-nilary ring and satisfies any one of the
following
conditions then it is a nilary ring.\\
(1) $A$ is a left (resp. a right) Noetherian.\\
(2) The sum of any collection of nilpotent ideals of $A$ is nilpotent.
\end{corollary}

\begin{proof}
By~\cite[p.~374]{Basic1999}, and by \cite[Proposition~1.4(i)]{Birkenmeier2013133}, (1) implies
(2), (2) gives that $A$ is a nilary ring and hence any one of the above
conditions gives that $A$ is a nilary ring.
\end{proof}

\begin{corollary} Let $A$ be a ring with unity.\\
If $A$ is a completely nilary ring then either $char(A)=0$ or
$char(A)=p^\beta$ for some positive integer $\beta,$ and $p$ is prime.
\end{corollary}


\section{WEAKLY NILARY}

In this section we introduce a new class of rings, called weakly
(p-)nilary ideal.
\begin{definition}  Let $A$ be a ring and let $I$ be a proper ideal of $A.$\\
We say that the ideal $I$ is a \textit{\textbf{weakly (p-)nilary ideal}} if
whenever $J$ and $K$ are (principal) ideals of $A$ with $0\neq JK\subseteq
I,$ then either $J^{m}\subseteq I$ or $K^{n}\subseteq I$ for some positive
integers $m$ and $n$ depending on $J$ and $K.$
\end{definition}

Clearly every (p-)nilary ideal is a weakly (p-)nilary ideal. But the converse
is not true, as the following example tell us.

\begin{example}Let $A=\mathbb{Z}_{6}$ and $I=\{0\}\unlhd A.$
Note that $I$ is weakly nilary but it is not nilary.
\end{example}

\begin{proposition}Let $I$ an ideal of $A.$ Then:\\
If $A$ is a (p-)nilary ring and $I$ is a weakly (p-)nilary ideal of $A$  then
$I$ is a (p-)nilary ideal of $A.$
\end{proposition}

\begin{proof}Let $J, K\unlhd A,$ with $JK\subseteq I.$ Now we have two cases.

Case 1: If $JK\neq0.$ Since $I$ is a weakly (p-)nilary ideal of $A$ then
$J^n\subseteq I$ or $K^m\subseteq I$ for some positive integers $m$ and $n,$
this implies that $I$ is nilary.

Case 2: If $JK=0.$ Since  $A$ is a (p-)nilary ring, then either  $J^t=0$ or
$K^s=0$ for some positive integers $t$ and $s.$ Hence $J^t=0\subseteq I$ or
$K^s=0\subseteq I$ and so $J^t\subseteq I$ or $K^s\subseteq I.$  Therefore
$I$ is (p-)nilary ideal of $A.$
\end{proof}

\begin{proposition}Let $I$ be a weakly (p-)nilary ideal of a ring $A.$ Then
either $I^2=0$ or $I$ is a (p-)nilary ideal of $A.$
\end{proposition}

\begin{proof} Assume that $I^2\neq0.$ Let $J, K\unlhd A,$ with $JK\subseteq I.$ We have $JK\neq0,$
then either $J^n\subseteq I$ or $K^m\subseteq I$ for some positive integers
$m$ and $n,$ this implies that $I$ is nilary. So assume that $JK=0,$ now we
have either $JI\neq0$ or $JI=0.$

Case 1: Assume that $JI\neq0,$ then there exists $x\in I$ such that
$Jx\neq0,$ implies that $J\langle x\rangle\neq0.$ Now $0\neq J\langle
x\rangle=J\langle x\rangle+0=J\langle x\rangle+JK=J(\langle
x\rangle+K)\subseteq I$ ( Indeed $J\langle x\rangle\subseteq J\cap\langle
x\rangle\subseteq\langle x\rangle\subseteq I$ and  $x\in I$ ). Hence $0\neq
J(\langle x\rangle+K)\subseteq I.$  Since $I$ is a weakly (p-)nilary ideal,
then $J^t\subseteq I$  or $(\langle x\rangle+K)^s\subseteq I$  for some
positive integers $t$ and $s,$ this implies that $J^t\subseteq I$  or
$K^s\subseteq I$ ( note that $K\subseteq(\langle x\rangle+K)$ implies
$K^s\subseteq(\langle x\rangle+K)^s$ ).

Case 2:  Assume that $JI=0.$ If $IK\neq0,$  then there exists $y\in I$ such
that $yK\neq0$ implies $\langle y\rangle K\neq0.$ Now $0\neq\langle y\rangle
K=\langle y\rangle K+0=\langle y\rangle K+JK=(\langle y\rangle+J)K\subseteq
I,$ hence $\neq0(\langle y\rangle+J)K\subseteq I,$ and since $I$ is a weakly
(p-)nilary ideal, then $(\langle y\rangle+J)^\alpha\subseteq I$ or
$K^\beta\subseteq I$  for some positive integers $\alpha$ and $\beta.$ If
$IK=0.$ Since $I^2\neq0,$ there exist $w, z\in I$ such that $wz\neq0.$
implies $0\neq\langle w\rangle\langle z\rangle.$ Then $$(\langle
w\rangle+J)(\langle z\rangle+K)=\langle w\rangle\langle z\rangle+\langle
w\rangle K+J\langle z\rangle+JK.$$ Now $JK=0.$ Also $\langle w\rangle K=0$
because $IK=0.$ And $J\langle z\rangle=0$ because $JI=0.$ But $\langle
w\rangle\langle z\rangle\neq0.$ Hence $$(\langle w\rangle+J)(\langle
z\rangle+K)=\langle w\rangle\langle z\rangle\subseteq I,$$ implies that
$$0\neq(\langle w\rangle+J)(\langle z\rangle+K)\subseteq I,$$ and since $I$
is a weakly (p-)nilary ideal, then $(\langle w\rangle+J)^\gamma\subseteq I$
or $(\langle z\rangle+K)^\lambda\subseteq I$ for some positive integers
$\gamma$ and $\lambda.$ Hence $J^\gamma\subseteq I$ or $K^\lambda\subseteq
I.$ Therefore $I$ is a (p-)nilary ideal.
\end{proof}

\begin{corollary}
Let $A$ be a semiprime ring with $I\unlhd A$. Then the following two conditions are equivalent:\\
(i) $I$ is a weakly (p-)nilary ideal;\\
(ii) $I=0$ or $I$ is a (p-)nilary ideal.
\end{corollary}

Next, we prove that weakly (p-)nilary ideals can be characterized by left as well as
right ideals.

\begin{proposition}\label{prop-right-p}Let $A$ be a ring with unity. The following conditions are equivalent:\\
(i) $L$ is a weakly (p-)nilary ideal of $A;$\\
(ii) let $I, J$ be (principal) right ideals of $A$ with $0\neq IJ\subseteq L,$ then
 either $I^n\subseteq L$ or $J^m\subseteq L$ for some $n,m\in\mathbb{N};$\\
(iii) let $I, J$ be (principal) left ideals of $A$ with $0\neq IJ\subseteq L,$ then
either $I^n\subseteq L$ or $J^m\subseteq L$ for some $n,m\in\mathbb{N}.$
\end{proposition}

\begin{proof}
(i)$\Rightarrow$(ii), let $I, J$ be (principal) right ideals of $A$ and
$0\neq IJ\subseteq L.$ Hence $$(AI)(AJ)=A(IA)J=A(I)J=AIJ\subseteq AL=L.$$ Since
$IJ\neq0$ then $AIJ\neq0$ this implies that $(AI)(AJ)\neq0.$ Now, since
$L$ is a weakly (p-)nilary ideal, then either $(AI)^n\subseteq L$ or $(AJ)^m\subseteq
L.$ If $(AI)^n\subseteq L$ for some $n,m\in\mathbb{N},$ then
$$I^n\subseteq(AI)^n.$$ Hence $I^n\subseteq L,$ similarly, $J^m\subseteq L.$

(ii)$\Rightarrow$(i) Clearly. Similarly (iii)$\Leftrightarrow$(i).
\end{proof}

\bibliographystyle{amsplain}
\bibliography{xbib}

\end{document}